\documentclass[a4paper,article,oneside]{memoir}
\usepackage{amsmath,amsfonts,mathrsfs,amssymb,amsthm,algorithm}
\usepackage[utf8]{inputenc}
\usepackage[T1]{fontenc}
\usepackage{paralist,enumitem}
\usepackage{xparse}
\usepackage[noend]{algpseudocode}
\usepackage{tikz}
\usepackage{lmodern,csquotes}
\usepackage[colorlinks=true,linkcolor=black,citecolor=blue]{hyperref}
\usepackage{xspace}
\usepackage{xcolor}
\newcommand{\seq}[2]{(#1_{#2})_{#2\geqslant0}}

\newcommand{\addpoint}[1]{#1\ ---\ }

\setsecnumdepth{subsubsection}

\setsecnumformat{\csname the#1\endcsname---}
\setsubsubsechook{\setsecnumformat{{\normalfont\sffamily\csname
 the##1\endcsname\ }}}
\setsubsechook{\setsecnumformat{\csname the##1\endcsname\ ---\ }}
\setsechook{\setsecnumformat{\csname the##1\endcsname---}}

\setsecheadstyle{\centering\Large\bfseries\sffamily}
\setsubsecheadstyle{\large\bfseries\sffamily}
\setsubsubsecheadstyle{\bfseries\itshape\addpoint}
\setsubsubsecindent{1em}
\setbeforesubsubsecskip{.5em plus .2em minus -.1em}
\setaftersubsubsecskip{-0em}
\setsubparaheadstyle{\itshape}
\newtheoremstyle{dfn}%
 {1.5ex plus .3ex minus .1ex}%
 {1ex plus .3ex minus .1ex}%
 {}%
 {}%
 {\sffamily}%
 {\,---\,}%
 {0em}%
 {$\bullet$\hbox{\ }#1\hbox{\ }#2}%

\newtheoremstyle{thm}%
 {1.5ex plus .3ex minus .1ex}%
 {1ex plus .3ex minus .1ex}%
 {\itshape}%
 {}%
 {\sffamily}%
 {\,---\,}%
 {0em}%
 {$\bullet$\hbox{\ }#1\hbox{\ }#2}%

\theoremstyle{dfn}
\newtheorem{dfntn}{Definition}[section]

\theoremstyle{thm}
\newtheorem{thrm}[dfntn]{Theorem}

\newtheorem{lmm}[dfntn]{Lemma}
\newtheorem{prpstn}[dfntn]{Proposition}

\newtheoremstyle{note}%
 {1ex plus .3ex minus .1ex}%
 {1ex plus .3ex minus .1ex}%
 {}%
 {}%
 {\itshape}%
 {.}%
 {1em}%
 {}%
\theoremstyle{note}

\newtheorem{remark}[dfntn]{Remark} 

\newcommand{\M}{\mathcal{M}}
\newcommand{\tq}{\colon}
\newcommand{\unit}{\mathbf{e}}
\newcommand{\R}{\mathcal{R}}
\DeclareDocumentCommand{\size}{ m g }{%
{\IfNoValueT{#2}{|#1|}\IfNoValueF{#2}{|#1|_{#2}}}%
}
\newcommand{\C}{\mathscr{C}}
\newcommand{\Cstar}{\mathfrak{C}}
\newcommand{\ZZ}{\mathbb{Z}}

\newcommand{\RR}{\mathbb{R}}
\newcommand{\RRR}{\mathscr{I}}

\newcommand{\U}{\mathsf{U}}

\newcommand{\UUU}{\mathcal{U}}
\newcommand{\G}{G}
\newcommand{\myleq}{\preccurlyeq}

\newcommand{\up}[1]{\uparrow #1}

\newcommand{\BM}{\partial\M}
\newcommand{\Mbar}{\overline\M}
\newcommand{\Pyr}[1]{\text{\normalfont\sffamily Pyr}_{\Sigma}(a_1)}
\newcommand{\Lk}{\mathscr{L}}

\newcommand{\pr}{\mathbb{P}}
\newcommand{\B}{\mathbf{B}}
\newcommand{\iid}{{\normalfont i.i.d.}\xspace}
\newcommand{\ie}{{\normalfont i.e.,}\xspace}
\newcommand{\Rref}[1]{\ref{#1}}
\newcommand{\ve}{\varepsilon}

\begin{document}

\begin{center}
 {\huge\bfseries Uniform generation of large traces}\par\medskip
 \begin{minipage}[t]{.35\linewidth}
 \begin{center}
 {\Large Samy Abbes}\\
 {Université Paris Cité\\
IRIF CNRS UMR 8243\\
8 place Aurélie Nemours\\
F-75013 Paris, France}
\end{center}
\end{minipage}
\qquad
\begin{minipage}[t]{.35\linewidth}
 \begin{center}
{\Large Vincent Jugé}\\ 
Univ Gustave Eiffel\\
CNRS, LIGM\\
F-77454 Marne-la-Vallée, France
 \end{center}
\end{minipage}
\end{center}

%
%
\bigskip
\begin{abstract}
We introduce an algorithm for the uniform generation of infinite traces, \ie infinite words up to commutation of some letters.
The algorithm outputs on-the-fly approximations of a theoretical infinite trace, the latter being distributed according to the exact uniform probability measure.
The average size of the approximation grows linearly with the time of execution of the algorithm, hence its output can be effectively used while running. 

Two versions of the algorithm are given. A version without rejection has a good production speed, provided that some precomputations have been done, but these may be costly. A version with rejection requires much fewer computations, at the expense of a production speed that can be small.

We also show that, for some particular trace monoids, one or the other version of the algorithm can actually be very good: few computations for a good production speed.

\medskip
\textbf{Keywords:} Random generation, trace monoid, Möbius polynomial, chordal graph
\end{abstract}
%
%

\section*{Introduction} 

\paragraph*{\sffamily Context and motivations.}
\label{sec:context-motivations}

Trace monoids are models of discrete-event concurrent systems.
Consider an alphabet $\Sigma$ equipped with a binary, symmetric and reflexive relation~$\R$, and let $I=(\Sigma\times\Sigma)\setminus\R$.
The \emph{trace monoid} $\M=\M(\Sigma,\R)$ is the presented monoid $\langle\Sigma \mid \RRR\rangle^+$ where $\RRR$ is the collection of pairs $(ab,ba)$ for $(a,b)$ ranging over~$I$.
Hence, an element in~$\M$, called a \emph{trace}, is the congruence class of some word $x\in\Sigma^*$, and congruent words are obtained from $x$ by successively exchanging the places of adjacent letters $a$ and $b$ such that $(a,b)\in I$.
For $(a,b)\in I$, the corresponding elements $a$ and $b$ are therefore commutative in~$\M$, which corresponds, from the point of view of systems theory, to the concurrency of actions represented by $a$ and~$b$.
Trace monoids have been ubiquitous in computer science and in combinatorics, since their very first use as models for databases with concurrency~\cite{diekert90,diekert95}.

A trace $x$ of a trace monoid $\M$ represents an execution of some concurrent system.
It is thus natural to inquire about the random generation of traces, and more precisely the random generation of traces of large length---the \emph{length} of a trace is simply the length of the associated congruent words.
Given a large integer~$N$, one could turn toward Boltzmann generation techniques to operate the random generation of traces of length~$N$.
However, when this is done, this technique is of little help for the generation of traces of a larger length; that would require to start again the procedure back from the beginning.
Whereas, when seeking for the simulation of ``real-life'' executions of a system, there is often little argument for stopping the execution at a particular time. 

It is therefore more appealing to design techniques for the random generation of infinite executions.
The precise target is the following: given a notion of infinite traces and a uniform measure for the space of these infinite traces, we look for an algorithm that produces, for each integer~$n$, a finite random trace of length proportional to~$n$ on average, and which coincides with a prefix of a uniformly distributed infinite trace. 

\medskip
\paragraph*{\sffamily The direct approach using the Cartier-Foata normal form.}
\label{sec:question-fo-cartier}

It has been well known since the work of Cartier and Foata~\cite{cartier69} that, for every trace monoid~$\M$, there is a finite graph $(\Cstar,\to)$, where $\Cstar$ is a finite subset of~$\M$, such that every trace $x\in\M$ admits one unique factorization as a concatenation $c_1\cdot\ldots\cdot c_k$ of elements $c_i\in\Cstar$ satisfying $c_i\to c_{i+1}$ for all~$i$.
The sequence $(c_1,\ldots,c_k)$ thus defined is the \emph{normal form} of~$x$.
Infinite traces correspond to infinite sequences~$(c_i)_{i\geqslant1}$, still with $c_i\to c_{i+1}$ for all~$i$.

This normal form has a nice probabilistic interpretation: the sequence $(c_i)_{i\geqslant1}$ corresponding to an infinite trace drawn \emph{uniformly at random} happens to be a Markov chain, with values in~$\Cstar$, whose initial distribution and transition matrix can be given through explicit formulas~\cite{abbes15}.
These facts seem to pave the way for an easy procedure for the random generation of infinite traces: simply simulate the Markov chain~$(c_i)_{i\geqslant1}$!

This procedure turns out to be inadequate in several cases, for complexity reasons.
Indeed, the set $\Cstar$ is defined as the set of cliques of the graph $(\Sigma,I)$, hence efficient computations indexed by $\Cstar$ in general are hopeless.
In fact, it all depends on the monoid~$\M$, and more precisely on the size of the set of cliques~$\Cstar$.

\medskip

\paragraph*{\sffamily Contributions: an approach to the random generation of traces using pyramidal traces.}
\label{sec:altern-appr}

Decomposing infinite traces according to their Cartier-Foata normal has the advantage, from the random generation viewpoint, of yielding a Markovian scheme among a finite set.
In this paper, we make use of another and original decomposition of infinite traces.
Namely, we consider so-called pyramidal traces, which are of arbitrary size, and which play the role of \emph{excursions} in classical probabilistic processes.
In particular, infinite random traces correspond to the infinite concatenation of an \iid\ sequence of pyramidal traces.

Considering concatenations of pyramidal traces amounts to trading the Markov scheme attached to the normal form of traces for a simpler \iid\ scheme. For this trade-off to be worth the cost, we need an efficient way of producing random pyramidal traces. This is the core of the random algorithms introduced in this paper. 

In order to generate random pyramidal traces, we introduce a random algorithm that is recursive with respect to the size of the alphabet of the trace monoid~$\M$. The algorithm applies to any irreducible trace monoid, which extends a work by one of the authors that was restricted to ``dimer-like'' trace monoids~\cite{abbes17}. In turn, one obtains an algorithmic procedure which runs endlessly and outputs at each unit of time a new random fragment of trace, such that the infinite concatenation of all these fragments would be an infinite trace uniformly distributed. The finite concatenations of the produced fragments are thus the prefixes of a random uniform infinite trace.

Two versions of the random algorithm for producing prefixes of infinite traces are given, one with a rejection procedure and one without rejection. On the one hand, the version without rejection produces prefixes at a linear average rate, and with a reasonable rate. However it may require up to exponentially many computations for on-the-fly tuning of the algorithm with the adequate probabilistic parameters; these computations could also be done before the execution of the algorithm, but that would then require of course an exponential amount of memory.

On the other hand, the version with rejection also outputs prefixes at a linear average rate, and requires only few computations. But the bound on the average production rate can be exponentially small.

It turns out that for several classes of graphs $(\Sigma,\R)$, one can guarantee much better bounds than the general bounds. This is the case for both versions of the algorithms. We show in particular a remarkable fact by characterizing precisely the graphs for which the version with rejection actually yields no rejection at all. This is the class of chordal graphs; for them, using the version with rejection is much advised, since then one wins on both sides: only few computations and a good production speed. For graphs with bounded tree-width, it is the version without rejection which is advised, since the precomputations are in a polynomial growth instead of the general exponential growth. Again, one wins on both sides with these graphs: a good production speed for few computations.

The complete procedure that we study, since it runs in infinite time, could hardly qualify as an algorithm, obviously. Yet, it can be effectively used for the production of \emph{finite} traces, since the procedure does indeed produce growing outputs during the execution, and at a linear rate in average. In particular, it can be used for verifying and quantifying properties that can be checked at finite horizon. Furthermore, since the computation keeps going on endlessly---contrary to a classical Boltzmann-like procedure, for instance---the procedure can be used even without \emph{a priori} bounds on the size of the expected sample, which makes it instrumental.

\medskip
Hence our contribution is threefold.
We first describe thoroughly the laws of pyramidal traces decomposing finite traces; for this, we make use of combinatorial identities on pyramidal traces found in the literature~\cite{viennot86}. We also give a precise proof of the decomposition of traces through pyramidal traces, which was kind of a folklore result until now.

Secondly, we introduce two versions of a recursive random algorithm for the generation of increasing prefixes of infinite traces, we prove their correctness and we give bounds on their average complexity. 

Finally, we discuss some classes of graphs for which the general bounds on the complexity of the algorithms can be much strengthened.  

\medskip
\paragraph*{\sffamily Outline of the paper.}

In Section~\Rref{sec:backgr-trace-mono}, we introduce the basic combinatorial and probabilistic material for trace monoids.
We also introduce infinite traces and define the uniform distribution at infinity.
Finally, we describe the probability laws of pyramidal traces decomposing finite traces.
Section~\Rref{sec:rand-gener-finite} contains the random algorithms, both for the generation of finite traces and of infinite traces.
Finally, Section~\Rref{sec:when-should-pyram} discusses the bounds of the algorithms for some special classes of trace monoids.

\section{Trace monoids and probability distributions}
\label{sec:backgr-trace-mono}

\subsection{Combinatorics of traces and discrete distributions}
\label{sec:comb-trac-pyram} 

Let~$\M=\M(\Sigma,\R)$ be a trace monoid.
It is known that elements of~$\M$ can be represented by \emph{heaps}~\cite{viennot86}, according to a bijective correspondence which we briefly recall now, following the presentation of~\cite{krattenthaler06}.
As illustrated in Fig.~\Rref{fig:pqjwdpoqjw}~(i), a \emph{heap} is a triple~$(P,\myleq,\ell)$, where~$(P,\myleq)$ is a poset and~$\ell:P\to\Sigma$ is a labeling of~$P$ by elements of~$\Sigma$, satisfying the two following properties:
\begin{inparaenum}[(1)]
\item\label{item:1} if~$x,y\in P$ are such that~$\ell(x) \R \ell(y)$, then, $x\myleq y$ or~$y\myleq x$;
\item the relation~$\myleq$ is the transitive closure of the relations from~(\Rref{item:1}).
\end{inparaenum}
More precisely, the heap is the equivalence class of~$(P,\myleq,\ell)$ up to isomorphism of labeled partial orders. 

To picture heaps corresponding to traces in~$\M$, one represents elements of~$\Sigma$ as elementary pieces that can be piled up with the following constraints, as illustrated in Fig.~\Rref{fig:pqjwdpoqjw}~(ii):
\begin{inparaenum}[(1)]
\item pieces can only be moved vertically;
\item pieces labeled by the same letter move along the same vertical lane; and
\item\label{itensoasaa} two pieces labeled by~$a$ and~$b$ in~$\Sigma$ can be moved independently of each other if and only if~$(a,b)\notin \R$.
\end{inparaenum}

\newcommand{\edge}[4]{ \draw[thick,->] (#1*0.75+#3*0.25,#2+0.25) -- (#1*0.25+#3*0.75,#4-0.25); }
\newcommand{\tracebox}[3]{
\node at (#1,#2) {$#3$};
\draw[thick] (#1-0.9,#2-0.4) -- (#1+0.9,#2-0.4) -- (#1+0.9,#2+0.4) -- (#1-0.9,#2+0.4) -- cycle;
}

\begin{figure}[t]
\centering
\begin{tikzpicture}[scale=0.5,>=stealth]
\node at (0,0) {$a$};
\node at (3,0) {$d$};
\node at (1,1) {$b$};
\node at (2,2) {$c$};
\node at (1,3) {$b$};
\node at (3,3) {$d$};
\node at (0,4) {$a$};

\edge{0}{0}{1}{1}
\edge{1}{1}{2}{2}
\edge{3}{0}{2}{2}
\edge{2}{2}{1}{3}
\edge{2}{2}{3}{3}
\edge{1}{3}{0}{4}

\node[anchor=north] at (1.5,-0.6) {(i)};

\begin{scope}[shift={(9,0)}]
\tracebox{0}{0}{a}
\tracebox{3}{0}{d}
\tracebox{1}{1}{b}
\tracebox{2}{2}{c}
\tracebox{1}{3}{b}
\tracebox{3}{3}{d}
\tracebox{0}{4}{a}

\draw[thick] (-1.2,-0.6) -- (4.2,-0.6);

\node[anchor=north] at (1.5,-0.6) {(ii)};
\end{scope}
\end{tikzpicture}
\caption{(i)~{\slshape Hasse diagram of the heap corresponding to the trace ${x=a\cdot b\cdot d\cdot c\cdot b\cdot a\cdot d}$ of the trace monoid~$\M(\Sigma,\R)$ where $\Sigma=\{a,b,c,d\}$ and where $\R$ is the reflexive and symmetric closure of~$\{(a,b),(b,c),(c,d)\}$.\quad}(ii)~{\slshape Graphical representation of the same heap with pieces piled up upon each other.}}
\label{fig:pqjwdpoqjw}
\end{figure}
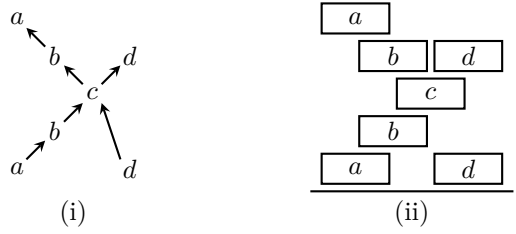

If~$\Sigma'$ is a subset of~$\Sigma$, we denote by~$\M_{\Sigma'}$ the sub-monoid of~$\M(\Sigma,\R)$ generated by~$\Sigma'$.
In particular,~$\M(\Sigma,\R)=\M_\Sigma$, a notation that we shall use from now on.

A \emph{clique} of $\M_\Sigma$ is any commutative product~$a_1\cdot\ldots\cdot a_k$, where $a_1,\ldots,a_k$ are distinct elements of~$\Sigma$ such that~$(a_i,a_j)\notin\R$ for all distinct~$i$ and~$j$.
We denote by~$\C_\Sigma$ the set of cliques of~$\M_\Sigma$\,.

Cliques of~$\M_\Sigma$ play an important role for the study of its combinatorics.
Indeed, each pair~$(\Sigma,\R)$ is associated with the \emph{Möbius polynomial}~$\mu_\Sigma(X)$ and the generating series~$G_\Sigma(X)$ defined as in~\cite{cartier69} by:
\begin{gather*}
\mu_\Sigma(X)=\sum_{\gamma\in\C_\Sigma}(-1)^{\size\gamma}X^{\size\gamma} \text{\quad and\quad }
G_\Sigma(X)=\sum_{x\in\M_\Sigma}X^{\size x}\,.
\end{gather*}

We generalize this series as follows.
Given a trace~$x\in\M_\Sigma$ that corresponds to a heap~$(P,\myleq,\ell)$, let~$\max(x)$ denote the set of labels of maximal elements of the poset~$(P,\myleq)$; equivalently, $\max(x)$ consists of those elements $a_k$ of $\Sigma$ for which $x$ can be written as a product $a_1 \cdot \ldots \cdot a_k$ of elements of~$\Sigma$.
For instance, if $x$ is the heap represented in Fig.~\Rref{fig:pqjwdpoqjw}, we have $\max(x) = \{a,d\}$.

It has been known since~\cite{cartier69} that the letters of $\max(x)$ commute pairwise, and that their commutative product $\hat{x}$ has the following property: for every heap~$y \in \M_\Sigma$, we have
\begin{gather}
\label{eq:000}
\max(y) \subseteq \max(x \cdot y) = \max(\hat{x} \cdot y) \subseteq \max(x) \cup \max(y).
\end{gather}

Now, let~$\U$ be a subset of~$\Sigma$.
We denote by~$\G_{\Sigma,\U}(X)$ the generating series of the elements $x$ of~$\M_\Sigma$ for which $\max(x) \subseteq \U$:
\begin{gather*}
\G_{\Sigma,\U}(X)=
\sum_{\substack{x\in\M_\Sigma\tq\max(x)\subseteq\U}}X^{\size x}\,. 
\end{gather*}

As observed in~\cite{viennot86}, every trace $x \in \M_\Sigma$ has a unique factorization $x = y \cdot z$ where $\max(y) \subseteq \U$ and $z \in \M_{\Sigma \setminus \U}$, thereby proving that:
\begin{gather}
\label{eq:1}
\G_{\Sigma,\U}(X)=\frac{\mu_{\Sigma\setminus\U}(X)}{\mu_{\Sigma}(X)}\,. 
\end{gather}

If~$\Sigma\neq\emptyset$, then $\mu_\Sigma(X)$ has a unique root of smallest modulus in the complex plane~\cite{goldwurm00,csikvari13,krob03}. This root, which is denoted by~$p_\Sigma$\,, is positive real and is at most~$1$. It coincides with the radius of convergence of the power series~$\G_{\Sigma,\U}(X)$ for any non empty subset~$\U$ of~$\Sigma$.
Hence, substituting~$p$ to~$X$ in the above identity provides an equality in~$\RR$ if~$p\in(0,p_\Sigma)$. As a particular case, obtained for~$\U=\Sigma$, one has:
\begin{gather}
\label{eq:5}
G_\Sigma(p)=\frac1 {\mu_\Sigma(p)} \qquad\text{for all~$p\in(0,p_\Sigma)$}. 
\end{gather}

Consequently, for each $p\in(0,p_\Sigma)$, the following formula defines a probability distribution on the countable set~$\M_\Sigma$:
\begin{gather}
\label{eq:7}
\forall x\in\M_\Sigma\quad B_{\Sigma,p}(x)=\mu_\Sigma(p)\,p^{\size x},
\end{gather}
where we simply write $B_{\Sigma,p}(x)$ instead of $B_{\Sigma,p}(\{x\})$ for a singleton set~$\{x\}$.

\subsection{Infinite traces and uniform measure at infinity}
\label{sec:infin-trac-unif} 

We briefly explain the construction of infinite traces and of the uniform measure on their set. 

The trace monoid $\M_\Sigma$ is equipped with its left-divisibility order~$\leqslant$\,, defined by $x\leqslant y$ if and only if there exists a trace $z\in\M_\Sigma$ such that $y=x\cdot z$; 
this trace $z$ is unique, and it is denoted by $x^{-1} \cdot y$. If $x\leqslant y$, we say that $x$ is a \emph{left divisor} of~$y$.
Note that the order $\leqslant$ is indeed a partial order.
Moreover, the partial order $(\M_\Sigma,\leqslant)$ is a lower semi-lattice: for each non-empty subset $X$ of~$\M_\Sigma$\,, there is a trace $x \in \M_\Sigma$, called the $\leqslant$-\emph{meet} of~$X$, such that
\begin{gather}
\label{eq:3b}
\forall y \in \M_\Sigma \quad y \leqslant x \iff (\forall z \in X \quad y \leqslant z).
\end{gather}
Furthermore, for each non-empty subset $X$ of $\M_\Sigma$ whose elements have a common upper $\leqslant$-bound, there exists a trace $x \in \M_\Sigma$, called the $\leqslant$-\emph{join} of~$X$, such that
\begin{gather}
\label{eq:3c}
\forall y \in \M_\Sigma \quad x \leqslant y \iff (\forall z \in X \quad z \leqslant y).
\end{gather}
This join coincides with the meet of the set of common upper $\leqslant$-bounds of~$X$.

Dually, since the relations in $\M_\Sigma$ are invariant under left-right reversal, $\M_\Sigma$~is also equipped with its right-divisibility order~$\geqslant$\,, defined by $x \geqslant y$ if and only if there exists a trace $z\in\M_\Sigma$ such that $x=z\cdot y$; the trace $z$ is unique, and it is denoted by $x \cdot y^{-1}$. If $x\geqslant y$, we say that $y$ is a \emph{right divisor} of~$x$.
Once again, the order $\geqslant$ is a partial order, each non-empty subset $X$ of $\M_\Sigma$ has a $\geqslant$-meet, and it has a $\geqslant$-join if its elements admit a common upper $\geqslant$-bound.

If $x=\seq xn$ and $y=\seq yn$ are two non-decreasing sequences in~$\M_\Sigma$\,, we define~$x\sqsubseteq y$ whenever, for all~$n \geqslant 0$,
there exists~$k \geqslant 0$ such that~$x_n\leqslant y_k$.
The relation~$\sqsubseteq$ is a preorder relation on the set of non-decreasing sequences.

Let~$\asymp$ be the equivalence relation defined by~$x\asymp y$ if and only if~$x\sqsubseteq y$ and~$y\sqsubseteq x$.
Equivalence classes of non-decreasing sequences modulo~$\asymp$ are called \emph{generalized traces}, and their set is denoted by~$\Mbar_\Sigma$\,.
The set~$\Mbar_\Sigma$ is equipped with an ordering relation, denoted by~$\leqslant$, which is the collapse of the preordering relation~$\sqsubseteq$.

The partial order $(\M_\Sigma,\leqslant)$ is embedded into~$(\Mbar_\Sigma,\leqslant)$, by sending an element $x\in\M_\Sigma$ to the equivalence class of the constant sequence~$\seq xn$ with~$x_n=x$ for all $n\geqslant0$.
Hence, we identify~$\M_\Sigma$ with its image in~$\Mbar_\Sigma$\,, and we put $\BM_\Sigma=\Mbar_\Sigma\setminus\M_\Sigma$\,.
Elements of~$\BM_\Sigma$ are called \emph{infinite traces}. Visually, infinite traces can be pictured as heaps obtained as in Fig.~\Rref{fig:pqjwdpoqjw}, but with infinitely many pieces piled up.

For every~$x\in\M_\Sigma$\,, we define the \emph{visual cylinder} of base~$x$ as the following subset of~$\BM_\Sigma$\,:
\[\up x=\{\xi\in\BM_\Sigma\tq x\leqslant\xi\}.\]

\emph{Via} the embedding~$\M_\Sigma\to\Mbar_\Sigma$\,, the family~$(B_{\Sigma,p})_{p \in (0,p_\Sigma)}$ can be seen as a family of discrete distributions on the compactification~$\Mbar_\Sigma$ rather than on~$\M_\Sigma$\,.
Standard techniques from functional analysis allow to prove the weak convergence of~$B_{\Sigma,p}$\,, when~$p\to p_\Sigma$\,, toward a probability measure~$\B_\Sigma$ on~$\BM_\Sigma$\,, characterized by the following identities~\cite{abbes15,abbes15conf}:
\begin{gather}
\label{eq:9}
\forall x \in \M_\Sigma\quad\B_\Sigma(\up x)=(p_\Sigma)^{\size x}\,.
\end{gather} 

\begin{dfntn}
\label{def:3}
The probability measure~$\B_\Sigma$ on\/~$\BM_\Sigma$ is called the \emph{uniform measure at infinity}.
\end{dfntn}

So far, we have thus defined a family of probability measures~$B_{\Sigma,p}$ on~$\Mbar_\Sigma$\,, for~$p$ ranging over the open interval~$(0,p_\Sigma)$, completed by a probability measure~$\B_\Sigma$.
Note the alternative:~$B_{\Sigma,p}$~is concentrated on~$\M_\Sigma$ for all $p<p_\Sigma$; whereas $\B_\Sigma$ is concentrated on~$\BM_\Sigma$\,.

\subsection{Pyramidal traces}
\label{sec:pyramidal-traces}

Given a trace monoid $\M=\M(\Sigma,\R)$, we define the \emph{link}~$\Lk(a_1)$ of a letter $a_1\in\Sigma$ by:
\begin{gather}
\label{eq:3}
\Lk(a_1)=\bigl\{b\in\Sigma\tq(a_1,b)\in\R\bigr\}.
\end{gather}

From the definition of the Möbius polynomials, we can thus derive the following identity, which will be useful later:
\begin{gather}
\label{eq:6}
\mu_\Sigma(X)=\mu_{\Sigma\setminus\{a_1\}}(X)-X\mu_{\Sigma\setminus\Lk(a_1)}(X)
\end{gather}

\begin{dfntn}
For $a_1\in\Sigma$, the \emph{$a_1$-pyramidal traces} are the traces belonging to the following subset:
\begin{gather}
\label{eq:11}
\Pyr{a_1}=\bigl\{z\cdot a_1\tq z\in\M_{\Sigma\setminus\{a_1\}} \text{ and } \max(z)\subseteq\Lk(a_1)\bigr\}.
\end{gather}
\end{dfntn}

Equivalently, and thanks to~\eqref{eq:000}, these are the traces $x$ with exactly one occurrence of the letter~$a_1$ and such that $\max(x) = \{a_1\}$.
This is illustrated in~Fig.~\Rref{fig:oiqoqihwq}~(i).

Now, for a generic trace~$x$, the successive occurrences of~$a_1$ within $x$ are associated with~$a_1$-pyramidal elements, as shown by the following decomposition result.

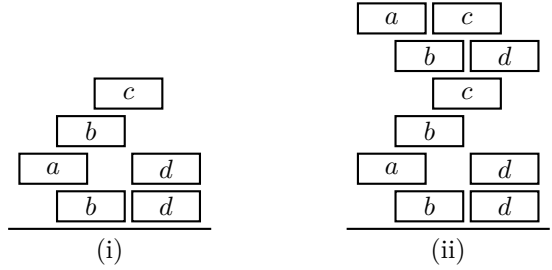
\begin{figure}[t]
\centering
\begin{tikzpicture}[scale=0.5,>=stealth]
\tracebox{1}{0}{b}
\tracebox{3}{0}{d}
\tracebox{0}{1}{a}
\tracebox{3}{1}{d}
\tracebox{1}{2}{b}
\tracebox{2}{3}{c}

\draw[thick] (-1.2,-0.6) -- (4.2,-0.6);

\node[anchor=north] at (1.5,-0.6) {(i)};

\begin{scope}[shift={(9,0)}]
\tracebox{1}{0}{b}
\tracebox{3}{0}{d}
\tracebox{0}{1}{a}
\tracebox{3}{1}{d}
\tracebox{1}{2}{b}
\tracebox{2}{3}{c}
\tracebox{1}{4}{b}
\tracebox{3}{4}{d}
\tracebox{0}{5}{a}
\tracebox{2}{5}{c}

\draw[thick] (-1.2,-0.6) -- (4.2,-0.6);

\node[anchor=north] at (1.5,-0.6) {(ii)};
\end{scope}
\end{tikzpicture}
\caption{(i)~\textsl{Heap representing the~$c$-pyramidal element~$b\cdot a\cdot b\cdot d\cdot d\cdot c$ in the trace monoid~$\M_\Sigma$\,, where~$(\Sigma,\R)$ is as in Figure~\Rref{fig:pqjwdpoqjw}.}\quad (ii)~\textsl{An element~$x\in\M_\Sigma$ whose decomposition through~$c$-pyramidal traces is~$(babddc)\cdot(bdc)\cdot a$.}}
\label{fig:oiqoqihwq}
\end{figure} 

\begin{prpstn}
\label{prop:3}
Let~$\M_\Sigma$ be a trace monoid and let~$a_1\in\Sigma$.
Every trace~${x\in\M_\Sigma}$ admits a unique factorization as a concatenation $u_0 \cdot \ldots \cdot u_k$, where $u_0,\dots,u_{k-1}$ are $a_1$-pyramidal traces and $u_k \in \M_{\Sigma\setminus\{a_1\}}$.

The integer $k$ is given by $k = |x|_{a_1}$, \ie $k$~is the number of occurrences of $a_1$ in~$x$.
Furthermore, $\max(x) = \max(u_k) \cup \{a_1\}$ if $k \geqslant 1$ and $u_k \in \M_{\Sigma \setminus \Lk(a_1)}$, and $\max(x) = \max(u_k)$ otherwise.
\end{prpstn}

\begin{proof}
First, if such a decomposition exists, each factor $u_0,\ldots,u_{k-1}$ contains exactly one occurrence of $a_1$, and $u_k$ contains no occurrence of $a_1$, so that $k = |x|_{a_1}$.

If $|x|_{a_1} = 0$, we have no choice but to take $u_0 = x$, in which case $\max(x) = \max(u_k)$.



Now, we prove the existence and uniqueness of this factorization by induction on~$|x|_{a_1}$; we have already treated the case $|x|_{a_1} = 0$, hence we assume that $|x|_{a_1} \geqslant 1$.
Let $\UUU_k$ be the set of all right divisors of $x$ that belong to~$\M_{\Sigma\setminus\{a_1\}}$.
This set admits a $\geqslant$-join, because $x$ is a common upper $\geqslant$-bound of its elements; let $u_k$ be this $\geqslant$-join.
If $\max(x \cdot u_k^{-1})$ contained a letter $b \neq a_1$, the trace $b \cdot u_k$ would also belong to~$\UUU_k$, a contradiction; hence $\max(x\cdot u_k^{-1})\subseteq\{a_1\}$.

Since $x$ contains an occurrence of~$a_1$, it cannot coincide with~$u_k$. 
This implies that $\max(x \cdot u_k^{-1}) \neq\emptyset$ and thus $\max(x\cdot u_k^{-1})=\{a_1\}$. It follows that $x \geqslant a_1 \cdot u_k$.
Moreover, the trace $y = x \cdot (a_1 \cdot u_k)^{-1}$ must satisfy the relation $\max(y) \subseteq \Lk(a_1)$; indeed, if $\max(y)$ contained a letter $b \notin \Lk(a_1)$, the trace $b \cdot u_k$ would also belong to~$\UUU_k$.
The induction hypothesis ensures that the trace $y = x \cdot (a_1 \cdot u_k)^{-1}$, which contains $k-1$ occurrences of the letter~$a_1$, admits a unique factorization $u_0 \cdot \ldots \cdot u_{k-2} \cdot \hat{u}_{k-1}$, where $\hat{u}_{k-1} \in \M_{\Sigma\setminus\{a_1\}}$.
But then, since $y \geqslant \hat{u}_{k-1}$, we know that $\max(\hat{u}_{k-1}) \subseteq \max(y) \subseteq \Lk(a_1)$, thanks to~(\Rref{eq:000}).
Consequently, the trace $u_{k-1} = \hat{u}_{k-1} \cdot a_1$ is $a_1$-pyramidal, and the desired factorization of $x$ is $u_0 \cdot \ldots \cdot u_k$.

Finally, let $v_0 \cdot \ldots \cdot v_k$ be a factorization of $x$ that satisfies the requirements of Proposition~\Rref{prop:3}.
Since $v_k$ belongs to $\UUU_k$, we must have $u_k \geqslant v_k$.
Let $z = u_k \cdot v_k^{-1}$, and let $z' = v_0 \cdot \ldots \cdot v_{k-1} = u_0 \cdot \ldots \cdot u_{k-1} \cdot z$.
 Using~\eqref{eq:000}, we have $\max(z) \subseteq \max(z') =\max(v_0 \cdot \ldots \cdot v_{k-1}) = \{a_1\}$.
Since $u_k$ contains no occurrence of the letter~$a_1$, the trace $z$ must be empty, \ie $u_k = v_k$.
Thus, the induction hypothesis ensures that the factorization $v_0 \cdot \ldots \cdot v_{k-2} \cdot (v_{k-1} \cdot a_1^{-1})$ of the heap $y$ coincides with $u_0 \cdot \ldots \cdot u_{k-2} \cdot (u_{k-1} \cdot a_1^{-1})$.
This proves that $u_i = v_i$ for all $i \leqslant k$, which completes the proof.
\end{proof} 

For example, if~$a_1 = c$, the element~$x=b\cdot a\cdot b\cdot d\cdot d\cdot c\cdot b\cdot d\cdot a\cdot c$, represented in Fig.~\Rref{fig:oiqoqihwq}~(ii), is decomposed as the product~$u_0 \cdot u_1 \cdot u_2$ of $a_1$-pyramidal elements given by~$u_0=b\cdot a\cdot b\cdot d\cdot d\cdot c$,~$u_1=b\cdot d\cdot c$ and~$u_2=a$.

\subsection{Law of pyramidal traces}
\label{sec:law-pyramidal-traces}

Fix a real $p\in(0,p_\Sigma)$ and pick $x\in\M$ at random according to the discrete probability distribution $B_{\Sigma,p}$ introduced in~\eqref{eq:7}. Then the integer~$k$ and the tuple $(u_0,\ldots,u_{k})$ decomposing $x$ as in Prop.~\Rref{prop:3} become random elements. Let they be denoted by the capital letters $K$ and $(U_0,\ldots,U_K)$ to underline their random nature. It is natural to investigate the law of the tuple of random variables $(K,U_0,\ldots,U_K)$, and this is the topic of this section.

It will actually be useful to have even more precise information. Namely, given a non-empty subset $T\subseteq\Sigma$, we study the law of the tuple $(K,U_0,\ldots,U_K)$ conditionaly to the event $\{\max(\xi)\subseteq T\}$, with $\xi$ being a random variable distributed according to~$B_{\Sigma,p}$. In particular, considering $T=\Sigma$ yields the laws without conditioning.

We first justify the existence of these conditional laws with the following statement (proved below with Th.~\ref{thr:3b}).

\begin{prpstn}
\label{thr:3a}
Let $\M_\Sigma$ be a trace monoid and let $p\in(0,p_\Sigma)$, where $p_\Sigma$ is the unique root of smallest modulus of the Möbius polynomial of~$\M_\Sigma$.
For each subset $T$ of\/~$\Sigma$, we have
\begin{gather}
\label{eq:12}
B_{\Sigma,p}(\xi\in\M\tq\max(\xi)\subseteq T)>0.
\end{gather}
\end{prpstn}

\begin{proof}
Let $\unit$ be the empty trace, \ie the unit element of the trace monoid.
Observing that
$B_{\Sigma,p}(\xi\in\M\tq \max(\xi)\subseteq T)\geqslant B_{\Sigma,p}(\{\unit\}) = \mu_{\Sigma}(p) > 0$ for every set $T \subseteq \Sigma$ proves~\eqref{eq:12}.
\end{proof}

Hence, there exists a distribution $D_{\Sigma,T}$ on $\M_\Sigma$, defined by
\begin{gather}
\label{eq:10}
D_{\Sigma,T}(\xi)=B_{\Sigma,p}\bigl(\xi\;\big|\max(\xi)\subseteq T\bigr);
\end{gather}
in other words,
\begin{gather}
\label{eq:10'}
D_{\Sigma,T}(\xi) = \frac{B_{\Sigma,p}(\xi)}{B_{\Sigma,p}(\xi\in\M\tq\max(\xi)\subseteq T)}
\end{gather}
if $\max(\xi) \subseteq T$, and $D_{\Sigma,T}(\xi) = 0$ otherwise.

\begin{thrm}
\label{thr:3b}
Let $\M_\Sigma$ be a trace monoid and let $p\in(0,p_\Sigma)$, where $p_\Sigma$ is the unique root of smallest modulus of the Möbius polynomial of~$\M_\Sigma$.
Consider some letter $a_1 \in \Sigma$.
We have $\mu_{\Sigma\setminus\{a_1\}}(p)>0$.
Hence, we define the real number
\begin{gather}
\label{eq:4}
r=p\frac{\mu_{\Sigma\setminus\Lk(a_1)}(p)}{\mu_{\Sigma\setminus\{a_1\}}(p)} = 1-\frac{\mu_\Sigma(p)}{\mu_{\Sigma\setminus\{a_1\}}(p)}.
\end{gather}

Furthermore, if $a_1 \in T$ with $T\subseteq\Sigma$, let $\xi=U_0\cdot\ldots\cdot U_K$ be the decomposition given by Proposition~{\normalfont\Rref{prop:3}} of a random element $\xi\in\M$ distributed according to the law~$D_{\Sigma,T}$. Then:
\begin{enumerate}
\item\label{item:3} The random variable $K=|\xi|_{a_1}$ follows a geometric law of parameter~$r$:
\begin{gather}
\label{eq:2}
\forall k\in\ZZ_{\geqslant0}\quad D_{\Sigma,T}(K=k)=(1-r)r^k
\end{gather}
\item\label{item:4} For every $k \geqslant 0$, and conditionally on the event $\{K=k\}$:
\begin{enumerate}
\item\label{item:8} the variables $U_0,\ldots,U_{k-1},U_k$ are independent;
\item\label{item:9} the variables $U_0 \cdot a_1^{-1},\ldots,U_{k-1} \cdot a_1^{-1}$ are distributed in $\M_{\Sigma\setminus\{a_1\}}$ according to $D_{\Sigma\setminus\{a_1\},\Lk(a_1)}$;
\item the variable $U_k$ is distributed in $\M_{\Sigma\setminus\{a_1\}}$ according to $D_{\Sigma\setminus\{a_1\},T}$.
\end{enumerate}
\end{enumerate}
\end{thrm}

Aiming to prove Theorem~\Rref{thr:3b}, we first establish two lemmas.

\begin{lmm}
\label{lem:1}
Let\/ $\M_\Sigma$ be a trace monoid, let $a_1\in\Sigma$ and let $p\in(0,p_\Sigma)$.
Then, we have $\mu_{\Sigma\setminus\{a_1\}}(p)>0$.
Furthermore, let $\xi$ be a random element in $\M_\Sigma$ distributed according to~$B_{\Sigma,p}$\,.
Then:
\begin{gather}
\label{eq:16}
B_{\Sigma,p}\bigl(\size{\xi}{a_1}>0\bigr)=1-\frac{\mu_\Sigma(p)}{\mu_{\Sigma\setminus\{a_1\}}(p)}=p\frac{\mu_{\Sigma\setminus\Lk(a_1)}(p)}{\mu_{\Sigma\setminus\{a_1\}}(p)}\,.
\end{gather}
\end{lmm}

\begin{proof}
Since $p \in (0,p_\Sigma)$, we first observe that $G_{\Sigma\setminus\{a_1\}}(p)\leqslant G_\Sigma(p) < \infty$.
Thus, the formula~\eqref{eq:5}, applied with $\Sigma\setminus\{a_1\}$ instead of~$\Sigma$, proves that $\mu_{\Sigma\setminus\{a_1\}}(p)>0$.

Then, let $r=B_{\Sigma,p}\bigl(\size{\xi}{a_1}>0\bigr)$.
Taking into account~\eqref{eq:7} and then applying~\eqref{eq:5} with the trace monoid $\M_{\Sigma\setminus\{a_1\}}$ indicates that
\begin{gather*}
1-r=B_{\Sigma,p}\bigl(\size{\xi}{a_1}=0\bigr)=\sum_{x\in\M_{\Sigma\setminus\{a_1\}}} \mu_\Sigma(p) p^{\size x}=\frac{\mu_\Sigma(p)}{\mu_{\Sigma\setminus\{a_1\}}(p)}\,,
\end{gather*}
which is the left equality in~\eqref{eq:16}.
Finally, the right equality in~\eqref{eq:16} derives directly from the identity~\eqref{eq:6}.
\end{proof}

\begin{lmm}
\label{lem:2}
Let $L_0,L_1,\ldots,L_k$ be subsets of $\M_\Sigma$, and let ${\varphi \colon L_0 \times \ldots \times L_k \to \M_\Sigma}$ be defined by
$\varphi(\ell_0,\ldots,\ell_k) = \ell_0 \cdot \ldots \cdot \ell_k$.
Assume that $\varphi$ is injective, and let $H$ be the set $\varphi(L_0 \times \ldots \times L_k)$.
Then, for all $p \in (0,p_\Sigma)$, we have
\begin{gather*}
B_{\Sigma,p}(H) = \mu_\Sigma(p)^{-k} \prod_{i=0}^k B_{\Sigma,p}(L_i).
\end{gather*}
\end{lmm}

\begin{proof}
Since $\varphi$ is injective and the length function $\size{\cdot}$ is additive, and according to~\eqref{eq:7}, we have:
\begin{align*}
B_{\Sigma,p}(H) & = \sum_{x \in H} \mu_\Sigma(p) p^{\size{x}}
= \sum_{(\ell_0,\dots,\ell_k)\in L_0\times\dots\times L_k} \mu_\Sigma(p) p^{\size{\ell_0}}\cdot\ldots\cdot p^{\size{\ell_k}} \\
& = \mu_\Sigma(p)\cdot\prod_{i=0}^k\Bigl(\sum_{\ell\in L_i}p^{\size{\ell}}\Bigr)
= \mu_\Sigma(p)\cdot\prod_{i=0}^k \Bigl(\mu_\Sigma(p)^{-1} B_{\Sigma,p}(L_i)\Bigr). \qedhere
\end{align*}
\end{proof}

\begin{proof}[Proof of Theorem~\Rref{thr:3b}.]
Lemma~\Rref{lem:1} proves that $\mu_{\Sigma\setminus\{a_1\}}(p)>0$, hence that the real number $r$ in~\eqref{eq:4} is well-defined, and that the identity in~\eqref{eq:4} holds.

Assume now that $a_1 \in T$, and let $L=\{x\in\M_{\Sigma\setminus\{a_1\}}\tq\max(x)\subseteq T\}$\,.
Given an integer $k \geqslant 0$, Proposition~\Rref{prop:3} means that the mapping ${\varphi \colon \Pyr{a_1}^k \times L\to\M_\Sigma}$ defined by $\varphi(\ell_0,\dots,\ell_k)=\ell_0\cdot\ldots\cdot \ell_k$ is a bijection from $\Pyr{a_1}^k \times L$ onto the subset $\U_k$ of $\M_\Sigma$ defined by:
\begin{gather*}
\U_k=\bigl\{\xi\in\M_\Sigma\tq |\xi|_{a_1}= k\text{ and }\max(\xi)\subseteq T\bigr\}.
\end{gather*}

The definition~\eqref{eq:7} of~$B_{\Sigma,p}$\,, the characterization~\eqref{eq:11} of~$\Pyr{a_1}$, the formula~\eqref{eq:1} and the value for $r$ given in~\eqref{eq:4} indicate that:
\begin{align*}
B_{\Sigma,p}(\Pyr{a_1})&=\mu_\Sigma(p)\cdot\left(\sum_{{z\in\M_{\Sigma\setminus\{a_1\}}\tq
\max(z)\subseteq\Lk(a_1)}}p^{\size{z\cdot a_1}}\right)\\
&=\mu_\Sigma(p)\frac{p \, \mu_{\Sigma\setminus\Lk(a_1)}(p)}{\mu_{\Sigma\setminus\{a_1\}}(p)}=r\mu_\Sigma(p)\,.
\end{align*}

Therefore, applying Lemma~\Rref{lem:2}:
\begin{gather}
\label{eq:27}
B_{\Sigma,p}(\U_k)=\mu_\Sigma(p)^{-k} \bigl(B_{\Sigma,p}(\Pyr{a_1})\bigr)^k B_{\Sigma,p}(L) = r^k B_{\Sigma,p}(L)\,.
\end{gather}
Since $D_{\Sigma,T}(K=k)$ is proportional to~$B_{\Sigma,p}(\U_k)$, the result of point~\Rref{item:3} follows.

Furthermore, Proposition~\Rref{prop:3} proves that, for all $u_0,\dots,u_{k-1}\in\Pyr{a_1}$ and $u_k \in L$:
\begin{align*}
B_{\Sigma,p}(U_0=u_0,\dots,U_K=u_k\;\big|\;K=k\bigr)
& = \frac{B_{\Sigma,p}(u_0 \cdot \ldots \cdot u_k)}{B_{\Sigma,p}(\U_k)} \\
& = \frac{\mu_\Sigma(p)}{B_{\Sigma,p}(\U_k)}\cdot p^{\size{u_0}}\cdot\ldots\cdot p^{\size{u_k}}\,.
\end{align*}

The above conditional law has a product form, which shows that the $U_i$ are independent random variables.
It also shows that the laws of $U_i \cdot a_1^{-1}$ when $i \leqslant k-1$, whose support is $\{\xi \in M_{\Sigma\setminus\{a_1\}} \colon \max(\xi) \subseteq \Lk(a_1)\}$, are all proportional to~$p^{\size{u_i \cdot a_1^{-1}}}$; 
this means that $U_i \cdot a_1^{-1}$ is distributed according to $D_{\Sigma\setminus\{a_1\},\Lk(a_1)}$.
Similarly, the law of $U_k$, whose support is $\{\xi \in M_{\Sigma\setminus\{a_1\}} \colon \max(\xi) \subseteq T\}$, is proportional to~$p^{\size{u_k}}$, which means that $U_k$ is distributed according to $D_{\Sigma\setminus\{a_1\},T}$.
\end{proof}

\section{Random generation of traces}
\label{sec:rand-gener-finite}

\subsection{Random generation of finite traces}
\label{sec:rand-gener-finite-1}

As a first task, we consider the random generation of finite traces according to a probability distribution~$B_{\Sigma,p}$ for~$p\in(0,p_\Sigma)$.
In view of Proposition~\Rref{prop:3} and of Theorem~\Rref{thr:3b}, we aim at generating random pyramidal traces.
Since $a_1$-pyramidal traces are traces $\xi \cdot a_1$ satisfying $\max(\xi)\subseteq\Lk(a_1)$, it is natural to introduce constraints of the form $\{\max(\xi)\subseteq T\}$ for non-empty subsets $T$ of~$\Sigma$.
And since $a_1$-pyramidal traces have only one occurrence of~$a_1$, the core of the generation concerns traces upon the alphabet $\Sigma\setminus\{a_1\}$.
Hence, our procedure is recursive on the size of the alphabet.
We obtain the following result.

\begin{thrm}
\label{thr:1}
Let $\M_\Sigma$ be a trace monoid.
Let $S$ and $T$ be subsets of\/~$\Sigma$, and let $p \in (0,p_S)$ be a real number.
Algorithm~1 below is a random recursive algorithm which, provided with the input~$(p,S,T)$, outputs an element~$\xi \in \M_\Sigma$ distributed according to the law~$D_{S,T}$, \ie to the law~$B_{S,p}( \cdot \mid \max(\xi) \subseteq T)$.

Furthermore, assume that each real number $\mu_X(p)$ has been precomputed;
that choosing an element in the intersection of two subsets of\/~$\Sigma$, or performing a function call, variable assignation and multiplication in~$\M_\Sigma$ takes a constant number of steps;
and that the call to a routine outputting a random integer~$X$ takes a number of steps bounded by~$X$.
Then, Algorithm~1 outputs an element~$\xi$ of\/~$\M_\Sigma$ in~$\mathcal{O}(\size{\Sigma} \bigl(\size{\xi}+1)\bigr)$ steps.
\end{thrm}

\begin{algorithm}
 \label{algo:1}
 \caption{Outputs~$\xi\in\M_S$ distributed according to~$D_{S,T}$}
\begin{algorithmic}[1]
\Require{Real parameter $p \in (0,p_S)$, subsets~$S$ and~$T$ of~$\Sigma$}
\If{$S \cap T = \emptyset$}
\State
\Return{$\unit$}\Comment{$\unit$ is the unit element of the monoid}
\Else
\State
\textbf{choose}~$a_1 \in S \cap T$
\State
$r\gets{1-\mu_S(p)/\mu_{S\setminus\{a_1\}}(p)}$
\State
$K\gets{\mathcal{G}(r)}$\Comment{Random integer with a geometric law}
\State
$\xi\gets\unit$\Comment{Initialization with the unit element of the monoid}
\For{$i=0$ to~$K-1$}
\State
$v_i\gets\text{output of Algorithm~1 on input } (p,S \setminus\{a_1\},\Lk(a_1))$
\State
$\xi\gets{\xi\cdot v_i\cdot a_1}$
\EndFor
\State
$v_K\gets\text{output of Algorithm~1 on input } (p,S \setminus\{a_1\},T)$
\State
$\xi\gets \xi\cdot v_K$
\State
\Return{$\xi$}
\EndIf
\end{algorithmic}
\end{algorithm}

Whereas the precomputation will appear as a bottleneck of this approach, our other assumptions are rather mild.
For instance, when~$|\Sigma| \leqslant 64$, subsets of $\Sigma$ are typically encoded via $64$-bit masks;
hence, both computing the intersection of two subsets and selecting the least element of a non-empty set amount to performing one negation and two bit-wise \& operations.
For sets $\Sigma$ of larger size, however, one may simply multiply by $|\Sigma|$ the running time of Algorithm~1 and of all subsequent algorithms:
given some ordering of $\Sigma$, the element $a_1$ chosen in line~4 will be the least element of $S \cap T$.

Similarly, if the user simply accepts obtaining an arbitrary factorization of the trace $\xi$ she is choosing at random, representing trace by linked lists with elements in $\Sigma$ allows multiplying elements of $\M_\Sigma$ in constant time.

\begin{proof}
We proceed by induction on~$|S|$.
When $S \cap T = \emptyset$, the result is immediate.
Hence, we assume that $S \cap T$ contains an element~$a_1$, and that the result holds for any subset of~$S \setminus \{a_1\}$.

First, the random integer $K$ follows a geometric law whose parameter $r$ is given by~\eqref{eq:4}.
Then, once $K$ is set to a given integer~$k$, the random traces $v_0,\ldots,v_k$ are sampled independently of each other, according to the laws prescribed by Theorem~\Rref{thr:3b}.
Consequently, the random trace $\xi$ that is output by Algorithm~1 follows the law $D_{S,T}$, as intended.

Second, let $f(n,k)$ be the maximal number of steps executed by Algorithm~1 on input $(S,T)$ when outputting an element~$\xi$, where $n = \size{S}$ and $k = \size{\xi}$.
Under the assumptions of Theorem~\Rref{thr:1}, there exists a constant $\kappa$ such that, when drawing an integer~$K$ and traces $v_0,\ldots,v_K$, Algorithm~1 requires at most
\begin{gather}
 \label{eq:8}
\kappa(K+1) + \sum_{i=0}^K f(n-1,|v_i|) 
\end{gather}
steps.
Up to increasing~$\kappa$, we also safely assume that $f(0,k) \leqslant \kappa$ when $n = 0$, since in that case $S = T = \emptyset$.

In those conditions, we prove by induction on $n\geq0$ that
\[f(n,k) \leqslant \kappa(n+1)(k+1)\]
for all integers~$k$.
The result is immediate if $n = 0$. Provided that our induction hypothesis holds for~$n-1$, we assume the output of a trace $\xi$ with $k=|\xi|$ and $K=|\xi|_{a_1}$. Then $K\leq k$ on the one hand, and $k+1= \sum_{i=0}^K (|v_i|+1)$ on the other hand. Using both~(\Rref{eq:8}) and the induction hypothesis to bound $f(n,k)$ yields thus:
\[
\frac{f(n,k)}{\kappa}
\leqslant (K+1) + \sum_{i=0}^K n(|v_i|+1) \leqslant (k+1) + n (k+1) \leqslant (n+1)(k+1),\]
completing the induction step and the proof of Theorem~\Rref{thr:1}.
\end{proof}

\subsection{Random generation of finite traces, with rejection}
\label{sec:rand-gener-finite-2}

One weakness of Algorithm~1 is that, as mentioned in Theorem~\Rref{thr:1}, implementing it efficiently may require precomputing and storing exponentially many real numbers~$\mu_X(p)$.
Hence, we investigate alternative algorithms, which would require precomputing $|\Sigma|$ real numbers $\mu_X(p)$ only.

\begin{thrm}
\label{thr:1b}
Let $\M_\Sigma$ be a trace monoid, and let $a_1,a_2,\ldots,a_n$ be the elements of\/~$\Sigma$.
Let $k$ and $\ell$ be two integers such that ${0 \leqslant k < \ell \leqslant n}$; let\/ $\Sigma_k$ be the set\/ $\{a_1,a_2,\ldots,a_k\}$, and let $p \in (0,p_{\Sigma_k})$ be a real number.
Algorithm~2 below is a random recursive algorithm which, provided with the input~$(p,k,\ell)$, outputs an element~$\xi \in \M_\Sigma$ distributed according to the law~$D_{\Sigma_k,\Lk(a_\ell)}$, \ie to the law $B_{\Sigma_k,p}( \cdot \mid \max(\xi) \subseteq \Lk(a_\ell))$.

Furthermore, let $S(k,\ell)$ be the number of steps required by one execution of Algorithm~2 when outputting a trace $\xi$ on input $(p,k,\ell)$, and let $L(k,\ell)$ be the length of that trace~$\xi$.
Under the same assumptions as in Theorem~\Rref{thr:1}, the random variables $S(k,\ell)$ and $L(k,\ell)$ have finite variances, and their averages satisfy the relations
\begin{align}
\mathbb{E}[S(k,\ell)] & = \mathcal{O}\left(\frac{k}{\mu_{\Sigma_k}(p)}\right)
\text{ and} \label{eqn:S}\\
\mathbb{E}[L(k,\ell)] & = p \left(\frac{\mu'_{\Sigma_k \setminus \Lk(a_\ell)}(p)}{\mu_{\Sigma_k \setminus \Lk(a_\ell)}(p)} - \frac{\mu'_{\Sigma_k}(p)}{\mu_{\Sigma_k}(p)}\right).\label{eqn:L}
\end{align}
\end{thrm}

\newsavebox{\rc}\savebox{\rc}{Recursive call:}
\newlength{\rcw}\settowidth{\rcw}{\usebox{\rc}}

\begin{algorithm}
\label{algo:1b}
\caption{Outputs~$\xi\in\M_{\Sigma_k}$ distributed according to~$D_{\Sigma_k,\Lk(a_\ell)}$}
\normalfont
\begin{algorithmic}[1]
\Require{Real parameter $p \in (0,p_{\Sigma_k})$, integers $0 \leqslant k < \ell \leqslant n$}
\If{$\Sigma_k \cap \Lk(a_\ell) = \emptyset$}\label{l1}
\State \Return $\unit$\label{l2}
\Comment{$\unit$ is the unit element of the monoid}
\Else
\State $\xi\gets\unit$
\State $\Sigma_k' \gets$ connected component of $a_k$ in the graph $(\Sigma_k,\R)$\label{line:sigmak'}
\State $r \gets
\begin{cases}
 \bigl(1-\mu_{\Sigma_k}(p)/\mu_{\Sigma_{k-1}}(p)\bigr)&\text{if $\Sigma'_k\cap\Lk(a_\ell)\neq\emptyset$}\\
 0&\text{if $\Sigma'_k\cap\Lk(a_\ell)=\emptyset$}
\end{cases}$
\label{line:rmu}
\Repeat{:}\label{line:while1b}
\State $R \gets \mathcal{B}(r)$\label{lineR}
\Comment{Random integer with Bernoulli law}
\State \label{line:endwhile1b}$v_\infty\gets$ 
\begin{minipage}[t]{.3\textwidth}
 output of Algorithm~2\\ on input $(p,k-1,\ell)$
\end{minipage}
\Comment{\begin{minipage}[t]{\rcw}
 \usebox{\rc}\\
 \strut\hfill\llap{$\max(v_\infty)\subseteq\Lk(a_\ell)$}
 \end{minipage}}
 \Until $R = 0$ or $\bigl(a_k \in \Lk(a_\ell)$ or $v_\infty \notin \M_{\Sigma_{k-1} \setminus \Lk(a_k)}\bigr)$
 \Comment{Rejection step}
\If{$R = 1$}\label{line:ifr1}
\State $K \gets 1+\mathcal{G}(r)$\Comment{Random integer with geometric law, shifted by $1$}
\Else
\State $K \gets 0$
\EndIf
\For{$i=1$ to~$K$}\label{line:for1b}
\State $v_i\gets$ output of Algorithm~2 on input $(p,k-1,k)$
\State $\xi\gets \xi\cdot v_i\cdot a_k$\label{line:endfor1b}
\EndFor
\State $\xi\gets \xi\cdot v_\infty$\label{line:finishxi}
\State \Return{$\xi$}
\EndIf
\end{algorithmic}
\end{algorithm}

\begin{proof}
We first prove the correctness of Algorithm~1 by induction on $k$.
Since the result is immediate when the set $\Sigma_k \cap \Lk(a_\ell)$ is empty, we assume that it is not empty, and that $k \geqslant 1$.

If $\Sigma_k\cap\Lk(a_\ell)=\emptyset$, then $\unit$ is the only trace $\xi\in\M_{\Sigma_k}$ satisfying $\max(\xi)\subseteq\Lk(a_\ell)$, whence the lines~\Rref{l1} and~\Rref{l2}.

We assume thus that $\Sigma_k\cap\Lk(a_\ell)\neq\emptyset$. Observe that the submonoids $\M_{\Sigma'_k}$ and $\M_{\Sigma_k \setminus \Sigma'_k}$ commute with each other, where $\Sigma'_k$ is defined at line~\Rref{line:sigmak'}. Hence, if $\Sigma'_k\cap\Lk(a_\ell)=\emptyset$, every trace $\xi \in \M_{\Sigma_k}$ such that $\max(\xi) \subseteq \Lk(a_\ell)$ must belong to the submonoid $\M_{\Sigma_k \setminus \Sigma'_k}$.
In particular, such a trace $\xi$ contains no occurrence of the letter~$a_k$, and the factorization into $a_k$-pyramids and a trace in $\M_{\Sigma_{k-1}}$ given in Proposition~\Rref{prop:3} is reduced to its rightmost factor; in turn, the algorithm goes as follows in that case: $r\gets0$ at line~\Rref{line:rmu}, $R\gets0$ at line~\Rref{lineR} and $\xi\gets v_\infty$ at line~\Rref{line:finishxi}. By the induction hypothesis, the trace output is distributed according to the law $D_{\Sigma_{k-1},\Lk(a_\ell)}$, which coincides with $D_{\Sigma_k,\Lk(a_\ell)}$ in that case.

On the contrary, assume that $\Sigma'_k \cap \Lk(a_\ell) \neq \emptyset$. The trace $v_\infty\in\M_{\Sigma_{k-1}}$ defined at line~\Rref{line:endwhile1b} already satisfies $\max(v_\infty)\subseteq\Lk(a_\ell)$; therefore, it follows from Proposition~\Rref{prop:3} that $v_\infty$~is \emph{not} eligible as a rightmost factor of a factorization into $K$ $a_k$-pyramids of a trace $\xi$ such that $\max(\xi)\subseteq\Lk(a_\ell)$ if and only if: $K>0$ and $a_k\notin\Lk(a_\ell)$ and $v_\infty\in\M_{\Sigma_{k-1}\setminus\Lk(a_k)}$.

Henceforth, Algorithm~2 outputs the same random trace $\xi$ as its following inefficient variant, which we call Algorithm~2b, where the lines~\Rref{line:ifr1} to~\Rref{line:finishxi} of the original algorithm have been included into the \textbf{repeat until} loop:
\begin{algorithmic}[1]
\setcounter{ALG@line}{6}
\Repeat{:}
\State $R \gets \mathcal{B}(r)$
\State $v_\infty\gets$ output of Algorithm~2 on input $(p,k-1,\ell)$
\If{$R = 1$}
\State $K \gets 1+\mathcal{G}(r)$
\Else
\State $K \gets 0$
\EndIf
\For{$i=1$ to~$K$}\label{line:fork}
\State $v_i\gets$ output of Algorithm~2 on input $(p,k-1,k)$
\State $\xi\gets \xi\cdot v_i\cdot a_k$
\EndFor
\State $\xi\gets \xi\cdot v_\infty$\label{line:xiend}
\Until $\max(\xi\cdot v_\infty)\subseteq\Lk(a_\ell)$
\label{line:while-test}
\State\Return{$\xi$}\label{returnfinal}
\end{algorithmic}

In line~\Rref{line:fork} of Algorithm~2b, the random integer $K$ follows a geometric law of parameter $r$.
Consequently, and according to Theorem~\Rref{thr:3b} applied with the sets $S = \Sigma_k$ and $T = (\Lk(a_\ell) \cap \Sigma_k) \cup \{a_k\}$, the random trace $\xi$ obtained in line~\Rref{line:xiend} of Algorithm~2 is distributed according to the law~$D_{\Sigma_k,T}$.
And the trace output at line~\Rref{returnfinal} is distributed according to the law $D_{\Sigma_k,T}\bigl(\cdot \mid \max(\xi) \subseteq \Lk(a_\ell)\bigr)$, \ie to the law $D_{\Sigma_k,\Lk(a_\ell)}$.

\medskip

Second, the law of the random variable $L(k,\ell)$ is entirely described by the generating series $G_{\Sigma_k,\Lk(a_\ell)}$.
More precisely, let $X$ be the set of traces $\xi \in \M_{\Sigma_k}$ such that $\max(\xi) \subseteq \Lk(a_\ell)$: we have
\[\mathbb{E}[L(k,\ell)] = \frac{\sum_{\xi \in X} |\xi| p^{|\xi|}}{\sum_{\xi \in X} p^{|\xi|}} = p \frac{G'_{\Sigma_k,\Lk(a_\ell)}(p)}{G_{\Sigma_k,\Lk(a_\ell)}(p)},\]
and~\eqref{eqn:L} is then a consequence of~\eqref{eq:1}.
Similarly,
\[\mathbf{Var}\bigl(L(k,\ell)\bigr) \leqslant \sum_{\xi \in X} |\xi|^2 p^{|\xi|} \leqslant p^2 G''_{\Sigma_k,\Lk(a_\ell)}(p) + p G'_{\Sigma_k,\Lk(a_\ell)}(p)<\infty,
\]
the later strict inequality holding since $p\in(0,p_\Sigma)$ and $p_\Sigma$ is the radius of convergence of the series~$G_{\Sigma_k,\Lk(a_\ell)}$\,.

\medskip

Third, we prove that $\mathbb{E}[S(k,\ell)] = \mathcal{O}(k) \mu_{\Sigma_{k-1}(p)}/\mu_{\Sigma_k}(p)$ and that $\mathbf{Var}(S(k,\ell))$ is finite by induction on $k$.
Since the result is immediate when $k = 0$, we assume that $k \geqslant 1$.

We first prove that $\mathbf{Var}(S(k,\ell))$ is finite.
Indeed, the number of iterations of the \textbf{repeat} loop of lines~\Rref{line:while1b} to~\Rref{line:endwhile1b}, say $\mathbf{X}$, follows a geometric law.
Moreover, each sampling of $v_\infty$ requires a number of steps that has finite variance, and these samplings are independent and identically distributed.
Consequently, the number of steps spent in that \textbf{repeat} loop also has finite variance.
We prove similarly that the \textbf{for} loop of lines~\Rref{line:for1b} to~\Rref{line:endfor1b} requires a number of steps with finite variance, and therefore that $S(k,\ell)$ itself has finite variance.

Then, we prove that
\begin{equation}
\label{eqn:SKLrec}
\mathbb{E}[S(k,\ell)] \leqslant \frac{\mathcal{O}(1) + \max\{\mathbb{E}[S(k-1,k)],\mathbb{E}[S(k-1,\ell)]\}}{1-r},
\end{equation}
where $r$ is defined in line~\Rref{line:rmu}.
To do so, we focus on the event $\mathcal{E}$, which is realized when $a_k \notin \Lk(a_\ell)$, and Algorithm~2, on input $(k-1,\ell)$, outputs a trace $\xi \in \M_{\Sigma_{k-1} \setminus \Lk(a_k)}$;
thus, if $a_k \in \Lk(a_\ell)$, we have $\mathbb{P}[\mathcal{E}] = 0$.

By discussing whether $R = 0$ or $\mathcal{E}$ is satisfied during the first recursive call to Algorithm~2 in line~\Rref{line:endwhile1b}, we observe that:
\begin{align*}
\mathbb{E}[S(k,\ell)] = \mathcal{O}(1) + \mathbb{E}[S(k-1,\ell)]
& + r \, \mathbb{P}[\mathcal{E}] \, \mathbb{E}[S(k,\ell)] \\
& + r \, \mathbb{P}[\overline{\mathcal{E}}] \, \mathbb{E}[K+1] \, (\mathcal{O}(1)+\mathbb{E}[S(k-1,k)]).
\end{align*}
Indeed, Algorithm~2 first consists in recursively calling Algorithm~2 on input $(k-1,\ell)$, in line~\Rref{line:endwhile1b}, and performing a small amount of calculations which accounts for the initial $\mathcal{O}(1) + \mathbb{E}[S(k-1,\ell)]$ term;
then, if $R = 1$ and $\mathcal{E}$ is realized, we shall simply start Algorithm~2 again;
if, however, $R = 1$ and $\mathcal{E}$ is not realized, we simply perform $K+1$ recursive calls to Algorithm~2 on input $(k-1,k)$, with a slight overhead at each call.

Since $K$ follows a geometric law $\mathcal{G}(r)$, we know that $\mathbb{E}[K+1] = 1/(1-r)$.
Thus,~\eqref{eqn:SKLrec} follows from the following inequalities:
\begin{align*}
\mathbb{E}[S(k,\ell)]
& = \frac{\mathcal{O}(1) + \mathbb{E}[S(k-1,\ell)] + \dfrac{r - r \, \mathbb{P}[\mathcal{E}]}{1-r} (\mathcal{O}(1)+\mathbb{E}[S(k-1,k)])}{1 - r \, \mathbb{P}[\mathcal{E}]} \\
& = \frac{\mathcal{O}(1)}{1-r} + \frac{(1 - r) \, \mathbb{E}[S(k-1,\ell)] + (r - r \, \mathbb{P}[\mathcal{E}]) \, \mathbb{E}[S(k-1,k)]}{(1 - r \, \mathbb{P}[\mathcal{E}]) \, (1-r)} \\
& \leqslant \frac{\mathcal{O}(1)}{1-r} + \frac{\max\{\mathbb{E}[S(k-1,\ell)],\mathbb{E}[S(k-1,k)]\}}{1-r}.
\end{align*}

Finally, if we set $\mathcal{S}(k) = \max\{\mathbb{E}[S(k,\ell)] \colon \ell > k\}$, equations~\eqref{eq:4} and~\eqref{eqn:SKLrec} show that
\[\mathcal{S}(k) \leqslant \frac{\mu_{\Sigma_{k-1}}(p)}{\mu_{\Sigma_k}(p)} \bigl(\mathcal{O}(1) + \mathcal{S}(k-1)\bigr) \leqslant \frac{\mathcal{O}(1)}{\mu_{\Sigma_k(p)}} + \frac{\mu_{\Sigma_{k-1}}(p)}{\mu_{\Sigma_k}(p)} \mathcal{S}(k-1),\]
which suffices to prove~\eqref{eqn:S} by induction on $k$.
\end{proof}

\subsection{Random generation of infinite traces}
\label{sec:rand-gener-infin}

In Sections~\Rref{sec:rand-gener-finite-1} and~\Rref{sec:rand-gener-finite-2}, we have seen how to generate random finite traces in $\M_\Sigma$ according to any distribution of the form~$B_{\Sigma,p}$.
In order to generate a random infinite trace according to the uniform distribution~$\B_\Sigma$, introduced in~\S~\Rref{sec:infin-trac-unif}, it is natural to investigate whether the simple stacking of finite traces could eventually produce a uniform infinite trace.
Unfortunately, this is not possible in general, for any parameter~$p$: see a counter-example in~\cite[\S~6.1.2]{abbes17}.

However, we shall see that the stacking of \emph{pyramidal} traces does indeed yield infinite traces uniformly distributed.
Let us first recall the notion of irreducible trace monoid.

\begin{dfntn}
\label{def:1}
If $\R$ is a binary reflexive and symmetric relation on $\Sigma$ such that the graph $(\Sigma,\R)$ is connected, then the associated trace monoid $\M_\Sigma$ is \emph{irreducible}.
\end{dfntn}

We reformulate results proved in \cite{abbes17} in the following statement.

\begin{thrm}
\label{thr:4}
Let $\M_\Sigma$ be an irreducible trace monoid, and let $a_1\in\Sigma$.
The distribution $h$ on $\Pyr{\Sigma}$ defined by
\begin{gather}
\label{eq:14}
h(v) = {p_\Sigma}^{|v|}
\end{gather}
for all $v\in\Pyr{\Sigma}$ is a probability distribution.

Let $(V_n)_{n\geqslant1}$ be an \iid\ sequence of random variables distributed with the law~$h$.
Then the random generalized trace $\xi$ defined by
\begin{gather*}
\xi=\bigvee_{n\geqslant1}(V_1\cdots V_n)
\end{gather*}
is an infinite trace distributed according to the uniform distribution $\B_\Sigma$ on~$\BM$.
\end{thrm}

This yields the following random generation algorithms.

\begin{thrm}
\label{thr:2}
Let $\M_\Sigma$ be an irreducible trace monoid, let $a_1,a_2,\ldots,a_n$ be the elements of\/~$\Sigma$, and let $a$ be some of the elements~$a_i$.
Then, Algorithm~1 is well-defined for all inputs~$(p_\Sigma,S,T)$ such that~$S \subseteq \Sigma\setminus\{a\}$, and Algorithm~2 is well-defined for all inputs~$(p_\Sigma,k,\ell)$ such that~$0 \leqslant k < \ell \leqslant n$.

Moreover, Algorithms~3 and~4 below are random endless procedures that output, at their~$k^{\text{th}}$ loop, an element~$\xi_k\in\M_\Sigma$ with the following properties:
\begin{enumerate}[nosep]
\item\label{item:5}~$(\xi_k)_{k\geqslant1}$ is a non-decreasing sequence.
\item\label{item:6} The element~$\xi=\bigvee_{k\geqslant1}\xi_k$, \ie the equivalence class of~$(\xi_k)_{k \geqslant 1}$ for the relation~$\asymp$, is an infinite trace distributed according to the uniform measure~$\B_\Sigma$.
\item\label{item:7} Under the same assumptions as in Theorem~\Rref{thr:1}, the first~$k$ loops of Algorithm~3 require the execution of~$\mathcal{O}(\size{\Sigma} \size{\xi_k})$ steps overall,
and the average and minimal sizes of\/~$\xi_k$ are linear in~$k$.
Hence, the algorithm produces on average a constant number of additional elements of\/~$\Sigma$ by unit of time, at a rate $\Omega(1/|\Sigma|)$.
\item Under these same assumptions, Algorithm~4 produces on average a constant number of additional elements of\/~$\Sigma$ by unit of time, at a rate $\Omega(\tau)$, where
\begin{equation}
\label{def:tau}
\tau = \left|\frac{\mu_{\Sigma_{n-1}}(p_\Sigma) + p_\Sigma \, \mu'_\Sigma(p_\Sigma)}{n}\right|.
\end{equation}
More precisely, when $t$ grows arbitrarily, the number of elements of $\Sigma$ produced after $t$ units of time is almost surely bounded from below by an expression of the form $\Omega(\tau) \, t$.
\end{enumerate}
\end{thrm}

\algtext*{Until}

\begin{algorithm}
\label{algo:2}
\caption{Outputs approximations of~$\xi\in\BM_\Sigma$ distributed according to~$\B_\Sigma$}
\normalfont
\begin{algorithmic}[1]
\Require{---}\Comment{No input}
\State
$\xi\gets\unit$\Comment{Initialization with the unit element of the monoid}
\Repeat{ \textbf{forever:}}
\State\label{repeat-start}
$v\gets\text{output of Algorithm~1 on input }(p_\Sigma,\Sigma\setminus\{a\},\Lk(a))$\label{repeat-call}
\State
$\xi\gets\xi\cdot v \cdot a$
\State\textbf{output}~$\xi$\Comment{Writes on a register}\label{repeat-end}
\Until{}
\end{algorithmic}
\end{algorithm}

\begin{algorithm}
\label{algo:2b}
\caption{Outputs approximations of~$\xi\in\BM_\Sigma$ distributed according to~$\B_\Sigma$}
\normalfont
\begin{algorithmic}[1]
\Require{---}\Comment{No input}
\State
$\xi\gets\unit$\Comment{Initialization with the unit element of the monoid}
\Repeat{ \textbf{forever:}}
\State\label{repeat-start:2b}
$v\gets\text{output of Algorithm~2 on input }(p_\Sigma,n-1,n)$\label{repeat-call:2b}
\State
$\xi\gets\xi\cdot v \cdot a_n$
\State\textbf{output}~$\xi$\Comment{Writes on a register}\label{repeat-end:2b}
\Until{}
\end{algorithmic}
\end{algorithm}

\begin{proof}
The inequalities $p_\Sigma<p_{\Sigma\setminus\{a\}}$ and $p_\Sigma < p_{\Sigma_{n-1}}$ rely on the irreducibility assumption on~$\M_\Sigma$, and follow from~\cite[Th.~4.5]{abbes17};
they prove that Algorithms~1 and~2 are well-defined when provided with the parameter $p = p_\Sigma$ as a part of their inputs.

We first focus on the correctness and efficiency of Algorithm~3.
Let $V_k$ be the (random) output of the $k$\textsuperscript{th} call to Algorithm~1 in line~\Rref{repeat-call}.
By construction, the $a$-pyramidal trace $V_k \cdot a$ is distributed according to a law proportional to~$p_\Sigma^{|V_k \cdot a|}$.
Hence, that law coincides with the law $h$ given in~\eqref{eq:14}.

Let $(\xi_k)_{k\geqslant1}$ be the sequence of outputs of Algorithm~3, and let $\xi=\bigvee_{k\geqslant1}\xi_k$.
Then $\xi_k=(V_1\cdot a)\cdot\ldots\cdot(V_k\cdot a)$.
Since the random variables $(V_k\cdot a)_{k\geqslant1}$ are \iid and follow the law $h$, Theorem~\Rref{thr:4} shows that $\xi$ is distributed according to~$\B_\Sigma$.

Moreover, Theorem~\Rref{thr:1} proves that every \textbf{repeat} block, when producing an element~$V$, requires executing $\mathcal{O}(\size{\Sigma} (\size{V}+1))$ steps.
Hence, producing the element $\xi_k$, whose length is positive, requires executing $\mathcal{O}(\size{\Sigma} \size{\xi_k})$ steps overall.

\medskip

Provided that $a = a_n$, Algorithms~1 and~2 output random traces with the same distribution, and thus the proof of correctness of Algorithm~3 applies verbatim to Algorithm~4.

Finally, since the calls to Algorithm~2 (in line~\Rref{repeat-call:2b} of Algorithm~4) are independent, the law of large numbers ensures that, when $k$ tends to $+\infty$, approximately $k \, \mathbb{E}[S(n-1,n)]$ steps will have been required, thereby producing a trace of size approximately $k \, \mathbb{E}[L(n-1,n)]$.
That is why, on average, Algorithm~4 produces a constant number of additional elements of $\Sigma$ per unit of time, at a rate $\Omega(\rho)$, where
\[\rho = \frac{\mathbb{E}[L(n-1,n)]}{\mathbb{E}[S(n-1,n)]}.\]

Furthermore, we have $p = p_\Sigma$, so that $\mu_\Sigma(p) = 0$.
Consequently, identities~\eqref{eq:6} and~\eqref{eqn:L} prove that
\begin{gather}
\label{eqn:Lsimple}
\mathbb{E}[L(n-1,n)] = \frac{-\mu_{\Sigma_{n-1}}(p_\Sigma) - p_\Sigma \mu'_\Sigma(p_\Sigma)}{\mu_{\Sigma_{n-1}}(p_\Sigma)}.
\end{gather}
Thus~\eqref{eqn:S} allows us to conclude that
\[\rho = \Omega\left(\frac{\mu_{\Sigma_{n-1}}(p_\Sigma)}{n} \, \frac{-\mu_{\Sigma_{n-1}}(p_\Sigma) - p_\Sigma \mu'_\Sigma(p_\Sigma)}{\mu_{\Sigma_{n-1}}(p_\Sigma)}\right) = \Omega(\tau). \qedhere\]
\end{proof}

One might be concerned by the fact that the sequence~$(\xi_k)_{k\geqslant1}$ that is output by Algorithms~3 and~4 has
a particular ``shape'', since it is the concatenation of~$a$-pyramidal traces.
For instance, if one wishes to use it for
parametric estimation or to sample some statistics on traces, the result could \emph{a priori} depend on the choice of~$a$.
But asymptotically, for a large class of statistics, the result will not depend on the choice of~$a$;
a precise justification of this fact can be found in~\cite{abbes16}.
For instance, the density of appearance of an arbitrary letter in an infinite trace can be approximated in this way.

\section{Which approach should be favored in practice for the random generation of large traces?}
\label{sec:when-should-pyram}

In Section~\Rref{sec:rand-gener-finite}, we proposed two algorithms, Algorithms~3 and~4, for generating increasing random sequences of finite traces whose limits are distributed according to the uniform measure $\B_\Sigma$ on infinite traces.
Both algorithms are based on decomposing traces into pyramids and, for a given trace monoid $\M_\Sigma$, the prefixes they output grow at constant speed.
However, each of them has its own strengths and weaknesses.

\subsection{Intrinsic problem difficulty}
\label{sec:intr-probl-diff}

First, what may look like a weakness common to both algorithms is the need to precompute quantities $\mu_S(p)$, where $\mu_S(X)$ is a Möbius polynomial and $S \subseteq \Sigma$.
Indeed, even computing the leading coefficient of such polynomials is \#P-hard~\cite{valiant79}, and thus one might be afraid of having to perform such computations. Furthermore, computing the values of $\mu_S(z)$ has been recently the topic of intense investigations~\cite{harvey18,galanis17}, showing that this computation can be extremely hard, depending on the size and the degree of the graph and the region of the parameter~$z$.

Our procedures for producing random finite traces rely on the precomputations of parameters of the form $\mu_S(p)$ for $p<p_\Sigma$. We claim that these computations are actually unavoidabe. To justify this claim, we appeal to a standard result from statistical estimation theory to prove the following:
\begin{compactenum}
\item[(\dag)]\label{item:3}\itshape
Assume given a procedure that produces a random finite traces distributed according to the law $B_{S,p}$, within a time proportional to the length of the produced trace. Let a precision $\ve>0$ and a confidence $\alpha\in(0,1)$ be given and let $\lambda>0$ be such that $\pr(|U|<\lambda)=\alpha$ where $U$ is a standard normal law. Then an asymptotic $\alpha$-confidence interval for $\mu_\Sigma(p)$  of length $\ve\min\{p,1-p\}$ is obtained in time $O\bigl(\frac{n^3\lambda^2}{\ve^2\min\{p,1-p\}}\bigr)$ for any $p<p_\Sigma$, and where $n=|\Sigma|$.
\end{compactenum}

\begin{proof}[Proof of\/~\normalfont(\dag)]
First, observe, using identity~\eqref{eq:6}, that
\begin{gather}
\label{eqn:mudec}
\mu_{U \setminus \{u\}}(p) \geqslant \mu_{U \setminus \{u\}}(p) - p \mu_{U \setminus \Lk(u)}(p) = \mu_U(p)
\end{gather}
whenever $U \subseteq \Sigma$ and $u \in U$, which means that the function $U \mapsto \mu_U(p)$ is non-increasing.
Since $U \setminus \Lk(u) \subseteq U \setminus \{u\}$, it even follows that
\begin{gather}
\label{eqn:muinc}
\mu_U(p) = \mu_{U \setminus \{u\}}(p) - p \mu_{U \setminus \Lk(u)}(p) \leqslant (1-p) \mu_{U \setminus \{u\}}(p)
\end{gather}
or, more generally, that $\mu_U(p) \leqslant (1-p) \mu_V(p)$ whenever $V \subseteq U \subseteq \Sigma$.

Now, let $a$ be an element of $\Sigma \setminus S$ that does not commute with all elements of~$S$, and let $T = S \setminus \Lk(a)$.
Identity~\eqref{eq:6} proves that $\mu_S(p) - p \mu_T(p) = \mu_{S \cup \{a\}}(p) \geqslant 0$, which yields the double inequality
\begin{gather}
\label{eqn:ratios}
p \leqslant \frac{\mu_S(p)}{\mu_T(p)} \leqslant 1-p.
\end{gather}
Observing that $\mu_S(p) / \mu_T(p) = B_{S,p}(\M_T)$ finally gives us a way to evaluate this ratio: just choose samples from $D_{S,S}$ and check whether they belong to~$\M_T$.
We can then, by induction, evaluate the real number $\mu_T(p)$ itself, from which we deduce our estimation of~$\mu_S(p)$. 

There are $n$ ratios to evaluate, each one being statistically estimated as the parameter of a Bernoulli law. Standard results from statistical estimation theory show that each ratio will be evaluated within a confidence interval of length~$\frac1n\ve\min\{p,1-p\}$, thanks to~\eqref{eqn:ratios}, after a number of trials  $O(\frac{n^2\lambda^2}{\ve^2\min\{p,1-p\}})$. The length of each trial is bounded in average by the expectation of~$|X|$, where $X$ is a trace that follows the law~$B_{S,p}$; whence the announced result.
\end{proof}

\begin{remark}
  A naïve approach would consist in testing on a sample $X_1,\ldots X_t$ of traces distributed according to $B_{S,p}$ whether $X_i=\unit$, since $B_{S,p}(\unit)=\mu_S(p)$; but this would lead to an estimation time within the order $O(\frac{\lambda^2}{\ve^2\mu_S(p)})$, and of course $\mu_S(p)$ can be drastically smaller than~$p$.
\end{remark}

\subsection{Pros and cons of Algorithms~1 to~4}
\label{sec:pros-cons-algorithms}

Second, why should we prefer Algorithms~1 and~3 over the state-of-the art approach used in~\cite{abbes15}?
This approach requires, for each pair $(c,c')$ of compatible cliques (\ie cliques such that $c \to c'$), to precompute and store the transition probability $p^{|c'|} \mu_{I(c)}(p) / \mu_{I(c')}(p)$, where $I(c)$ denotes the set $\Sigma \setminus \bigcup_{a_i \in c} \Lk(a_i)$ of monoid generators that commute with $c$ but do not divide $c$.
Thus, in case the trace monoid $\M_\Sigma$ admits $|\Cstar|$ cliques, we should store approximately $|\Cstar|^2$ real numbers.
This is often prohibitive, since $|\Cstar|$ can go up to $2^{|\Sigma|}$, and should be expected to be exponential in $|\Sigma|$.
Algorithms~1 and~3 aim at reducing this space complexity, by storing only $|\Cstar|$ real numbers: even if this quantity can still be exponential in $|\Sigma|$, this is a quadratic improvement over the previous solution.

Indeed, Algorithm~3 produces prefixes $\xi_k$ at a very good, constant speed.
Unfortunately, Algorithm~1, on which it relies, may require storing exponentially many real numbers $\mu_X(p)$, where $X$ is a subset of $\Sigma$.
This is because, when executing Algorithm~1 on a given input~$(p,S,T)$, it will choose some element $a_1$ of $S \cap T$ before recursively calling itself on inputs $(p,S \setminus \{a_1\}, \Lk(a_1))$ and $(p,S \setminus \{a_1\}, T)$.
These recursive calls will require choosing elements $a_2 \in S \cap \Lk(a_1) \setminus \{a_1\}$ and $a_2' \in S \cap T \setminus \{a_1\}$, and then accessing the real numbers $\mu_{S \setminus \{a_1,a_2\}}(p)$ and $\mu_{S \setminus \{a_1,a_2'\}}(p)$.
Thus, if $S \cap T \cap \Lk(a_1) = \emptyset$, we cannot choose $a_2 = a'_2$, and we must compute real numbers $\mu_X(p)$ for two distinct sets $X$; this unfortunate situation can happen again and again at each step of the computation, at least for a substantial number of recursive calls.

On the contrary, Algorithm~4, which is based on a \emph{rejection sampling} approach~\cite{robert99}, allows computing only a linear quantity of real numbers $\mu_X(p)$.
This circumvents the main disadvantage of Algorithm~3, by dramatically reducing its space complexity.
However, this decrease comes at the expense of a possibly much smaller production rate, which could be as small as the real number $\tau$ defined in~\eqref{def:tau}.
For instance, if $\M_\Sigma$ is the \emph{dimer} monoid, \ie if $\Sigma = \{a_1,a_2,\ldots,a_n\}$ and $\mathcal{R} = \{(a_i,a_j) \colon i-1 \leqslant j \leqslant i+1\}$, it can be proved that
\[\tau \sim \frac{n^2}{2^{n+2} \pi^2}\]
when $n \to +\infty$.
This suggests that our lower bound on the production rate of Algorithm~4 might be exponentially small:
although this might seem disappointing, such a trade-off where an exponential improvement in memory consumption is compensated by an exponential deterioration in time consumption is quite standard.

Yet, in practice, our actual production rates are not so small, and experiments suggest that they are even \emph{constant} on dimer monoid, in spite of this exponentially small lower bound.
On some other trace monoids, such as the $(2n+1)$-element \emph{star} monoid defined by $\Sigma = \{a_1,a_2,\ldots,a_{2n+1}\}$ and $\mathcal{R} = \{(a_i,a_j) \colon i = j \pm n \text{ or } (\min\{i,j\} = 1 \text{ and } i,j \leqslant n+1)\}$, it can be proved that
\[\tau \sim \frac{\ln(n)}{n^3}\]
when $n \to +\infty$, this lower bound being matched (up to a sub-linear factor) by the actual production rate of Algorithm~4.
In fact, the upper bound on $\mathbb{E}[S(k,\ell)]$ provided by~\eqref{eqn:S} is tight for some monoids, but far from tight for some other monoids;
we do not any family of irreducible trace monoids where the production rate of Algorithm~4 would be exponentially (or even super-polynomially small), which suggests that Algorithm~4 might be a good option anyway.

\subsection{Classes of graphs where our algorithms are efficient}

In conclusion, there is yet to discover an algorithmic procedure that would allow producing prefixes $\xi_k$ at a provably good (say, polynomially small in $|\Sigma|$) constant speed, while requiring to store only polynomially many real numbers $\mu_X(p)$.
However, there are large classes of trace monoids for which Algorithm~2 rarely proceeds to rejecting traces provided in recursive calls.
Such classes are characterized by the graph $\mathcal{G} = (\Sigma,\mathcal{R})$ associated with the monoid $\M_\Sigma$.

One such class on which Algorithms~1 and Algorithms~3 actually require no precomputing is the class of graphs with bounded \emph{tree-width}, a widely-studied class of ``tree-like'' graphs~\cite{cygan15}.
Indeed, when $\mathcal{G}$ has tree-width $k$, each Möbius polynomial $\mu_X$ can be computed in time exponential in $k$ but only polynomial in~$|\Sigma|$.
Consequently, instead of precomputing and storing exponentially many real numbers $\mu_X(p_\Sigma)$, we may simply recompute on-the-fly these numbers whenever needed.

Another class of interest for Algorithms~2 and~4 is the class of \emph{chordal graphs}~\cite{fulkerson1965}.
A graph $\mathcal{G}$ is \emph{chordal} if it contains no induced cycles of length $4$ or more: in other words, every cycle of length $4$ or more contains non-consecutive vertices that are neighbors of each other.
Indeed, the production rate of Algorithm~4 is very good when $\mathcal{G}$ is chordal, as outlined by the following result.

\begin{prpstn}
\label{prop:IG}
Let~$\M_\Sigma$ be a trace monoid and let $\mathcal{G} = (\Sigma,\mathcal{R})$ be the associated graph.
The graph $\mathcal{G}$ is chordal if and only if there exists an enumeration $a_1,a_2,\ldots,a_n$ of the elements of\/ $\Sigma$ for which Algorithms~2 and~4 will never require rejection.
\end{prpstn}

\begin{proof}
Monoids $\M_\Sigma$ on which Algorithms~2 and~4 never require rejection are those for which there exists an enumeration $a_1,\ldots,a_n$ of the elements of $\Sigma$ such that, whenever $1 \leqslant k < \ell \leqslant n$, either $a_k \in \Lk(a_\ell)$ or the set $\Sigma'_k$ computed in line~\Rref{line:sigmak'} is disjoint from $\Lk(a_\ell)$.

First, assume that $\M_\Sigma$ is such a monoid, and let $\mathcal{C}$ be a proper cycle of $\mathcal{G}$ with length $m \geqslant 3$.
Let $a_{i_1} < a_{i_2} < \ldots < a_{i_m}$ be the vertices of $\mathcal{C}$: vertices from $\mathcal{C}$ are neighbors of each other if and only if they are consecutive.
Since $a_{i_m}$ has neighbors in the sets $\Sigma'_{i_{m-1}}$ and $\Sigma'_{i_{m-2}}$, both $a_{i_{m-1}}$ and $a_{i_{m-2}}$ must be its neighbors;
similarly, $a_{i_{m-1}}$ has a neighbor in the set $\Sigma'_{i_{m-2}}$, and thus that neighbor is $a_{i_{m-2}}$.
Hence, $m = 3$, which means that $\mathcal{G}$ is chordal.

Conversely, if $\mathcal{G}$ is chordal, it is known that it has a \emph{perfect elimination ordering}~\cite{fulkerson1965}, \ie that there exists an enumeration $a_1,a_2,\ldots,a_n$ of the elements of $\Sigma$ such that, for all integers $k \leqslant n$, any two vertices $a_i, a_j \in \Lk(a_k)$ such that $i > j > k$ are neighbors of each other.
In particular, consider two integers $k < \ell$ such that $\Sigma'_k \cap \Lk(a_\ell) \neq \emptyset$.
Let $a_{i_1},a_{i_2},\ldots,a_{i_m}$ be a path of minimal length that goes from $a_{i_1} = a_k$ to $a_{i_m} = a_\ell$ while staying in $\Sigma_k \cup \{a_\ell\}$.
For each integer $j \in \{2,3,\ldots,m-1\}$, both vertices $a_{i_{j-1}}$ and $a_{i_{j+1}}$ belong to $\Lk(a_{i_j})$ while not being neighbors of each other, and thus $i_j > \min\{i_{j-1},i_{j+1}\}$.
Hence, an induction shows that $i_m > i_{m-1} > \ldots > i_1 = k$, and since $k \geqslant i_{m-1}$, it follows that $m = 2$, \ie that $a_k \in \Lk(a_\ell)$.
\end{proof}

In particular, when $\M_\Sigma$ is the dimer monoid with $n$ generators, Theorem~\Rref{thr:2} only proves a quite bad lower bound on the rate at which Algorithm~4 produces elements of $\Sigma$.
Proposition~\Rref{prop:IG} shows that, although that lower bound might be accurate if the elements of $\Sigma$ were enumerated in an inadequate order, enumerating these elements appropriately may result in much better efficiency guarantees. 

\bibliographystyle{plain}
\bibliography{biblio} 

\end{document}